\documentclass[a4paper,12pt, twoside, reqno]{amsart}
\usepackage{amsmath}%
\usepackage{amsfonts}%
\usepackage{amssymb}%
\usepackage{graphicx}
\newtheorem{theorem}{Theorem}
\theoremstyle{plain}

\newtheorem{corollary}{Corollary}

\newtheorem{definition}{Definition}
\newtheorem{example}{Example}

\newtheorem{remark}{Remark}

\newtheorem{ass}{Assumption}[section]
\numberwithin{equation}{section}
\newcommand{\R}{\mathbb R}

\begin{document}
\title{Coupled fixed point theorems for $\phi$-contractive mixed monotone mappings in partially ordered metric spaces}
\author{Vasile Berinde}

\begin{abstract}
In this paper we extend the coupled fixed point theorems for mixed monotone operators $F:X \times X \rightarrow  X$ obtained in  [T.G. Bhaskar, V. Lakshmikantham, \textit{Fixed point theorems in partially ordered metric spaces and applications}, Nonlinear Anal. \textbf{65} (2006) 1379-1393] and [N.V. Luong and N.X. Thuan, \textit{Coupled fixed points in partially ordered metric spaces and application}, Nonlinear Anal. \textbf{74} (2011) 983-992], by weakening the involved contractive condition. An example as well an application to nonlinear Fredholm integral equations are also given in order to illustrate the effectiveness of our generalizations.
\end{abstract}
\maketitle

\pagestyle{myheadings} \markboth{Vasile Berinde} {Coupled fixed point theorems for $\phi$-contractive}

\section{Introduction and preliminaries} 

The existence of fixed points and coupled fixed points for contractive type mappings in partially ordered metric spaces has been considered recently  by  several authors: Ran  and  Reurings  \cite{Ran},  Bhaskar  and  Lakshmikantham  \cite{Bha},  Nieto  and  Lopez  \cite{Nie06}, \cite{Nie07},  Agarwal  et  al.  \cite{Aga}, Lakshmikantham and Ciric \cite{LakC}, Luong and Thuan \cite{Luong}. These results found important applications to the study of matrix equations or ordinary differential equations and integral equations, see \cite{Ran}, \cite{Bha}, \cite{Nie06}, \cite{Nie07}, \cite{Luong} and references therein.

%Following basically the same approach as the one in \cite{Ran}, and unifying the results in \cite{Nie06} and \cite{Nie07}, Bhaskar and Lakshmikantham \cite{Bha} obtained some coupled fixed point results for mixed monotone operators of the form $F:X \times X \rightarrow  X$ satisfying a certain contractive type condition, where $X$ is a partially ordered metric space. 

In order to fix the framework needed to state the main result in \cite{Bha}, we remind the following notions. Let$\left(X,\leq\right)$ be a partially ordered set and  endow the product space $X \times X$ with the following partial order:$$\textnormal{for } \left(x,y\right), \left(u,v\right) \in X \times X,  \left(u,v\right) \leq \left(x,y\right) \Leftrightarrow x\geq u,  y\leq v.$$

We say that a mapping $F:X \times X \rightarrow X$ has the \textit{mixed monotone property} if $F\left(x,y\right)$ is monotone nondecreasing in $ x$ and is monotone non increasing in $y$, that is, for any $x, y \in X,$   
$$  x_{1}, x_{2} \in X,  x_{1} \leq  x_{2} \Rightarrow  F\left(x_{1},y\right) \leq  F\left(x_{2},y\right) $$
  and, respectively,
$$ y_{1},  y_{2} \in  X, y_{1} \leq y_{2} \Rightarrow F\left(x,y_{1}\right)  \geq F\left(x,y_{2}\right). $$

A pair $ \left(x,y\right) \in X \times X $ is called a \textit{coupled  fixed  point} of the mapping $F$ if                                              $$ F\left(x,y\right) = x,  F\left(y,x\right) = y.$$
The next theorem has been established  in \cite{Bha}.

\begin{theorem}[Bhaskar and Lakshmikantham \cite{Bha}]\label{th2}
	Let  $\left(X,\leq\right)$ be a partially ordered set and suppose there is a metric $d$  on $X$ such that $\left(X,d\right)$  is a complete metric space. Let $F : X \times X \rightarrow X $ be a continuous mapping having the mixed monotone property on $X$. Assume that there exists a constant $k \in \left[0,1\right)$  with                                                                                   
\begin{equation} \label{Bhas}
	d\left(F\left(x,y\right),F\left(u,v\right)\right) \leq \frac{k}{2}\left[d\left(x,u\right) + d\left(y,v\right)\right],\textnormal{ for each }x \geq u, y \leq  v.
\end{equation}  
If there exist $x_{0}, y_{0} \in X$ such that  
$$                                                                                       
x_{0} \leq F\left(x_{0},y_{0}\right)\textrm{ and }y_{0} \geq F\left(y_{0},x_{0}\right),
$$
 then there exist $x, y \in X$ such that $$x = F\left(x,y\right)\textnormal{ and }y = F\left(y,x\right).$$ 
	\end{theorem}        
	
As shown in \cite{Bha}, the continuity assumption of $F$ in Theorem \ref{th2} can be replaced by the following alternative condition imposed on the ambient space $X$:

\begin{ass} \label{1.1}
$X$ has the property that

(i) if a non-decreasing sequence $\{x_n\}_{n=0}^{\infty}\subset X$ converges to $x$, then $x_n\leq x$ for all $n$;

(ii) if a non-increasing sequence $\{x_n\}_{n=0}^{\infty}\subset X$ converges to $x$, then $x_n\geq x$ for all $n$;
\end{ass}
	
Bhaskar and Lakshmikantham \cite{Bha} also established uniqueness results for coupled fixed points and fixed points and illustrated these important results by proving the existence and uniqueness of the solution for a periodic boundary value problem. These results were then extended and generalized by several authors in the last five years, see  \cite{LakC}, \cite{Luong} and references therein. Amongst these generalizations, we refer to the ones obtained Luong and Thuan in \cite{Luong}, who have considered instead of \eqref{Bhas} the more general contractive condition
\begin{equation} \label{eq-Luong}
\varphi\left(	d\left(F\left(x,y\right),F\left(u,v\right)\right)\right) \leq \frac{1}{2}\varphi\left(d\left(x,u\right) + d\left(y,v\right)\right)-\psi\left(d\left(x,u\right) + d\left(y,v\right)\right)
\end{equation}
where $\varphi,\,\psi:[0,\infty)\rightarrow [0,\infty)$ are functions satisfying some appropriate conditions.

Note that for $\varphi(t)=t$ and $\psi(t)=\frac{1-k}{2}t$, with $0\leq k<1$, condition \eqref{eq-Luong} reduces to \eqref{Bhas}.

Starting from the results in \cite{Bha} and \cite{Luong}, our main aim in this paper is to obtain more general coupled fixed point theorems for mixed monotone operators $F : X \times X \rightarrow X $  satisfying a contractive condition which is significantly weaker that the corresponding conditions \eqref{Bhas} and \eqref{eq-Luong} in \cite{Bha} and \cite{Luong}, respectively. We also illustrate how our results can be applied to obtain existence and uniqueness results for integral equations under weaker assumptions than the ones in \cite{Luong}.

%spun ca eu nu folosesc subaditivitatea lui $\phi$ !!!

\section{Main results}

Let $\Phi$ denote the set of all functions $\varphi:[0,\infty)\rightarrow [0,\infty)$ satisfying 

$(i_\varphi)$ $\varphi$ is continuous and non-decreasing;

$(ii_\varphi)$ $\varphi(t)< t$ for all $t>0$,

\noindent and $\Psi$ denote the set of all functions $\psi:[0,\infty)\rightarrow [0,\infty)$ which satisfy

$(i_\psi)$ $\lim\limits_{t\rightarrow r}\psi(t)>0$ for all $r>0$ and $\lim\limits_{t\rightarrow 0+}\psi(t)=0$.

Examples of typical functions $\varphi$ and $\psi$ are given in \cite{Luong}, see also \cite{Ber07} and \cite{Rus2}.
	
The first main result in this paper is the following coupled fixed point theorem which generalizes Theorem 2.1 in \cite{Luong} and Theorem 2.1 in \cite{Bha}.

\begin{theorem}\label{th3}
	Let  $\left(X,\leq\right)$ be a partially ordered set and suppose there is a metric $d$  on $X$ such that $\left(X,d\right)$  is a complete metric space. Let $F : X \times X \rightarrow X $ be a  mixed monotone mapping for which there exist $\varphi\in \Phi$ and $\psi \in \Psi$  such that for all $x,y,u,v\in X$ with $x \geq u, y \leq  v$,
$$
\varphi\left(\frac{d\left(F\left(x,y\right),F\left(u,v\right)\right)+d\left(F\left(y,x\right),F\left(v,u\right)\right)}{2}\right) \leq 
$$                                                                       
\begin{equation} \label{Bhas1}
 \leq  \varphi\left(\frac{d\left(x,u\right) + d\left(y,v\right)}{2}\right)-\psi\left(\frac{d\left(x,u\right) + d\left(y,v\right)}{2}\right).
\end{equation}  
Suppose either

(a) $F$ is continuous or

(b) $X$ satisfy Assumption \ref{1.1}.

If there exist $x_{0}, y_{0} \in X$ such that  
\begin{equation} \label{mic}
 x_{0} \leq F\left(x_{0},y_{0}\right)\textrm{ and }y_{0} \geq F\left(y_{0},x_{0}\right),
\end{equation} or
\begin{equation} \label{mare} 
x_{0} \geq F\left(x_{0},y_{0}\right)\textrm{ and }y_{0} \leq F\left(y_{0},x_{0}\right),
\end{equation}
 then there exist $\overline{x}, \overline{y} \in X$ such that $$\overline{x} = F\left(\overline{x},\overline{y}\right)\textrm{and }\overline{y} = F\left(\overline{y},\overline{x}\right).$$ 
	\end{theorem}

\begin{proof}
Consider the functional $d_2:X^2\times X^2 \rightarrow \mathbb{R}_{+}$ defined by
$$
d_2(Y,V)=\frac{1}{2}\left[d(x,u)+d(y,v)\right],\,\forall Y=(x,y),V=(u,v) \in X^2.
$$
It is a simple task to check that $d_2$ is a metric on $X^2$ and, moreover, that, if $(X,d)$ is complete, then $(X^2,d_2)$ is a complete metric space, too.
%\begin{equation} \label{eq-2}
%\frac{1}{2}	\left[d\left(F\left(x,y\right),F\left(u,v\right)\right)+d\left(F\left(y,x\right),F\left(v,u\right)\right)\right] \leq k \frac{1}{2} \left[d\left(x,u\right) + d\left(y,v\right)\right].
%\end{equation}
Now consider the operator $T:X^2\rightarrow X^2$ defined by
$$
T(Y)=\left(F(x,y),F(y,x)\right),\,\forall Y=(x,y) \in X^2.
$$
Clearly, for $Y=(x,y),\,V=(u,v)\in X^2$, in view of the definition of $d_2$, we have
$$
d_2(T(Y),T(V))=\frac{d\left(F\left(x,y\right),F\left(u,v\right)\right)+d\left(F\left(y,x\right),F\left(v,u\right)\right)}{2}
$$
and
$$
d_2(Y,V)=\frac{d\left(x,u\right) + d\left(y,v\right)}{2}.
$$
Thus, by the contractive condition \eqref{Bhas1} we obtain that $F$ satisfies the following $(\varphi,\psi)$-contractive condition: %\eqref{Cond-Ran} appearing in Theorem \ref{th2}
\begin{equation} \label{contr}
\varphi\left(d_2(T(Y), T(V))\right)\leq \varphi\left(d_2(Y,V)\right)-\psi\left(d_2(Y,V)\right),\,\forall Y\geq V \in X^2.
\end{equation} 
%for the case of the complete metric space $(X^2,d_2)$.
%If we note that conditions \eqref{mic} and \eqref{mare} imply condition \eqref{exist} in Theorem \ref{th2}, then the conclusion follows directly by applying Theorem \ref{th2}, as any fixed point of $T$ is a coupled fixed point of $F$. 
Assume \eqref{mic} holds (the case \eqref{mare} is similar). Then, there exists $x_0,y_0\in X$ such that
$$
x_0\leq F(x_0,y_0) \textnormal{ and } y_0\geq F(y_0,x_0). 
$$
Denote $Z_0=(x_0,y_0)\in X^2$ and consider the Picard iteration associated to $T$ and to the initial approximation $Z_0$, that is, the sequence $\{Z_n\}\subset X^2$ defined by
\begin{equation} \label{eq-3}
Z_{n+1}=T (Z_n),\,n\geq 0,
\end{equation}
with $Z_n=(x_n,y_n)\in X^2,\,n\geq 0$.

Since $F$ is mixed monotone, we have
$$
Z_0=(x_0,y_0)\leq (F(x_0,y_0), F(y_0,x_0))=(x_1,y_1)=Z_1
$$
and, by induction,
$$
Z_n=(x_n,y_n)\leq (F(x_n,y_n), F(y_n,x_n))=(x_{n+1},y_{n+1})=Z_{n+1},
$$
which shows that the mapping $T$ is monotone and the sequence $\{Z_n\}_{n=0}^{\infty}$ is non-decreasing.
%\varphiWe follow now the steps in the proof of Banach's contraction fixed point theorem. 
Take $Y=Z_n\geq Z_{n-1}=V$ in \eqref{contr} and obtain
\begin{equation} \label{eq-4.2}
\varphi\left(d_2(T(Z_{n}),T(Z_{n-1})\right)\leq \varphi\left(d_2 (Z_n, Z_{n-1})\right)-\psi\left(d_2 (Z_n, Z_{n-1})\right),\,n\geq 1,
\end{equation}
which, in view of the fact that $\psi\geq 0$, yields
$$
\varphi\left(d_2(Z_{n+1},Z_n)\right)\leq \varphi\left(d_2 (Z_n, Z_{n-1})\right),\,n\geq 1,
$$
which, in turn, by condition $(i_\varphi)$ implies
\begin{equation} \label{eq-5}
d_2(Z_{n+1},Z_n)\leq d_2 (Z_n, Z_{n-1}),\,n\geq 1,
\end{equation}
and this shows that the sequence $\{\delta_n\}_{n=0}^{\infty}$ given by
$$
\delta_n=d_2 (Z_n, Z_{n-1}),\,n\geq 1,
$$
is non-increasing. Therefore, there exists some $\delta \geq 0$ such that
\begin{equation} \label{eq-4.1}
\lim_{n\rightarrow \infty} \delta_n=\frac{1}{2}\lim_{n\rightarrow \infty}\left[d(x_{n+1},x_n)+d(y_{n+1},y_n)\right]=\delta.
\end{equation}
We shall prove that $\delta=0$. Assume the contrary, that is, $\delta>0$. Then by letting $n\rightarrow \infty$ in \eqref{eq-4.2} we have
$$
\varphi(\delta)=\lim_{n\rightarrow \infty} \varphi(\delta_{n+1})\leq \lim_{n\rightarrow \infty} \varphi(\delta_{n})-\lim_{n\rightarrow \infty} \psi(\delta_{n})=
$$
$$
=\varphi(\delta)-\lim_{\delta_n\rightarrow \delta+} \psi(\delta_{n})<\varphi(\delta),
$$
a contradiction. Thus $\delta=0$ and hence
\begin{equation} \label{eq-4}
\lim_{n\rightarrow \infty} \delta_n=\frac{1}{2}\lim_{n\rightarrow \infty}\left[d(x_{n+1},x_n)+d(y_{n+1},y_n)\right]=0.
\end{equation}

We now prove that $\{Z_n\}_{n=0}^{\infty}$ is a Cauchy sequence in  $(X^2,d_2)$, that is, $\{x_n\}_{n=0}^{\infty}$ and $\{y_n\}_{n=0}^{\infty}$ are Cauchy sequences in $(X,d)$. Suppose, to the contrary, that at least one of the sequences $\{x_n\}_{n=0}^{\infty}$, $\{y_n\}_{n=0}^{\infty}$ is not a Cauchy sequence. Then there exists an $\epsilon>0$ for which we can find subsequences $\{x_{n(k)}\}$, $\{x_{m(k)}\}$ of $\{x_n\}_{n=0}^{\infty}$ and $\{y_{n(k)}\}$, $\{y_{m(k)}\}$ of $\{y_n\}_{n=0}^{\infty}$ with $n(k)>m(k)\geq k$ such that
\begin{equation} \label{eq-6}
\frac{1}{2}\left[d(x_{n(k)},x_{m(k)})+d(y_{n(k)},y_{m(k)})\right]\geq \epsilon,\,k=1,2,\dots.
\end{equation}
Note that we can choose $n(k)$ to be the smallest integer with property $n(k)>m(k)\geq k$ and satisfying \eqref{eq-6}. Then
\begin{equation} \label{eq-5.1}
d(x_{n(k)-1},x_{m(k)})+d(y_{n(k)-1},y_{m(k)})< \epsilon.
\end{equation}
By \eqref{eq-6} and \eqref{eq-5.1} and the triangle inequality we have
$$
\epsilon\leq r_k:=\frac{1}{2}\left[d(x_{n(k)},x_{m(k)})+d(y_{n(k)},y_{m(k)})\right]\leq
$$
$$
\frac{d(x_{n(k)},x_{n(k)-1})+d(y_{n(k)},y_{n(k)-1})}{2}+\frac{d(x_{n(k)-1},x_{m(k)})+d(y_{n(k)-1},y_{m(k)})}{2}
$$
$$ 
\leq \frac{d(x_{n(k)},x_{n(k)-1})+d(y_{n(k)},y_{n(k)-1})}{2}+\epsilon.
$$
Letting $k\rightarrow \infty$ in the above inequality and using \eqref{eq-4} we get
\begin{equation} \label{eq-7.1}
\lim_{k\rightarrow \infty} r_k:=\lim_{k\rightarrow \infty} \frac{1}{2}\left[d(x_{n(k)},x_{m(k)})+d(y_{n(k)},y_{m(k)})\right]=\epsilon.
\end{equation}
Since $n(k)>m(k)$, we have $x_{n(k)}\geq x_{m(k)}$ and $y_{n(k)}\leq y_{m(k)}$ and hence by \eqref{Bhas1}
$$
\varphi\left(r_{k+1}\right)=\varphi\left(\frac{1}{2}	\left[d\left(F\left(x_{n(k)},y_{n(k)}\right),F\left(x_{m(k)},y_{m(k)}\right)\right)+\right.\right.
$$
$$
 \left.\left.+d\left(F\left(y_{m(k)},x_{m(k)}\right),F\left(y_{n(k)},x_{n(k)}\right)\right)\right]\right) \leq \varphi\left(r_{k}\right)-\psi\left(r_{k}\right).
$$
Letting $k\rightarrow \infty$ in the above inequality and using \eqref{eq-7.1} we get
$$
\varphi(\epsilon)=\varphi(\epsilon)-\lim_{k\rightarrow \infty} \psi\left(r_{k}\right)= \varphi(\epsilon)-\lim_{r_k\rightarrow \epsilon+} \psi\left(r_{k}\right)<\varphi(\epsilon),
$$
a contradiction. This shows that $\{x_n\}_{n=0}^{\infty}$ and $\{y_n\}_{n=0}^{\infty}$ are indeed Cauchy sequences in the complete metric space $(X,d)$.

This implies there exist $\overline{x},\overline{y}$ in $X$ such that
$$
\overline{x} =\lim_{n\rightarrow \infty} x_n \textnormal{ and } \overline{y} =\lim_{n\rightarrow \infty} y_n.
$$
Now suppose that assumption (a) holds. Then
$$\overline{x} = \lim_{n\rightarrow \infty} x_{n+1}=\lim_{n\rightarrow \infty} F(x_n,y_n)=F\left(\overline{x},\overline{y}\right)
$$ 
and
$$
\overline{y} = \lim_{n\rightarrow \infty} y_{n+1}=\lim_{n\rightarrow \infty} F(y_n,x_n)=F\left(\overline{y},\overline{x}\right),
$$
which shows that $(\overline{x},\overline{y})$ is a coupled fixed point of $F$.

Suppose now assumption (b) holds. Since  $\{x_n\}_{n=0}^{\infty}$ is a non-decreasing sequence that converges to $\overline{x}$, we have that $x_n\leq \overline{x}$ for all $n$. Similarly, $y_n\geq \overline{y}$ for all $n$.

Then

$$
d(\overline{x},F(\overline{x},\overline{y}))\leq d(\overline{x},x_{n+1})+d(x_{n+1},F(\overline{x},\overline{y}))=d(\overline{x},x_{n+1})+
$$
$$
+d(F(x_n,y_n),F(\overline{x},\overline{y}))
$$
and
$$
d(\overline{y},F(\overline{y},\overline{x}))\leq d(\overline{y},y_{n+1})+d(y_{n+1},F(\overline{y},\overline{x}))=d(\overline{y},y_{n+1})+
$$
$$
+d(F(y_n,x_n),F(\overline{y},\overline{x})).
$$
So
$$
d(\overline{x},F(\overline{x},\overline{y}))- d(\overline{x},x_{n+1})\leq d(F(x_n,y_n),F(\overline{x},\overline{y}))
$$
and
$$
d(\overline{y},F(\overline{y},\overline{x}))- d(\overline{y},y_{n+1})\leq d(F(y_n,x_n),F(\overline{y},\overline{x}))
$$
and hence
$$
\frac{1}{2}\left[d(\overline{x},F(\overline{x},\overline{y}))- d(\overline{x},x_{n+1})+d(\overline{y},F(\overline{y},\overline{x}))- d(\overline{y},y_{n+1})\right]\leq
$$
$$
\leq \frac{1}{2}\left[d(F(x_n,y_n),F(\overline{x},\overline{y}))+d(F(y_n,x_n),F(\overline{y},\overline{x}))\right]
$$
which imply, by the monotonicity of $\varphi$ and condition \eqref{Bhas1},
$$
\varphi\left(\frac{1}{2}\left[d(\overline{x},F(\overline{x},\overline{y}))- d(\overline{x},x_{n+1})+d(\overline{y},F(\overline{y},\overline{x}))- d(\overline{y},y_{n+1})\right]\right)\leq
$$
$$
\leq \varphi\left( \frac{1}{2}\left[d(F(x_n,y_n),F(\overline{x},\overline{y}))+d(F(y_n,x_n),F(\overline{y},\overline{x}))\right]\right)\leq
$$
$$
\leq \varphi\left(\frac{d(x_n,\overline{x})+d(y_n,\overline{y})}{2}\right)-\psi\left(\frac{d(x_n,\overline{x})+d(y_n,\overline{y})}{2}\right).
$$
Letting now $n\rightarrow \infty$ in the above inequality, we obtain
$$
\varphi\left(\frac{d(\overline{x},F(\overline{x},\overline{y}))+d(\overline{y},F(\overline{y},\overline{x}))}{2}\right)\leq
\varphi\left(0\right)-\psi\left(0\right)=0,
$$
which shows, by $(ii_{\varphi})$, that $d(\overline{x},F(\overline{x},\overline{y}))=0$ and $d(\overline{y},F(\overline{y},\overline{x}))=0$. 

\end{proof}

 \begin{remark} \em
Theorem \ref{th3} is more general than Theorem 2.1 in \cite{Luong} and Theorem \ref{th2} (i.e., Theorem 2.1 in \cite{Bha}), since the contractive condition \eqref{Bhas1} is more general than \eqref{Bhas} and \eqref{eq-Luong}, a fact which is clearly illustrated by the next example.
 \end{remark}
 
\begin{example} \label{ex1} \em
Let $X=\mathbb{R},$ $d\left(x,y\right)=|x-y|$ and $F:X\times X \rightarrow X$ be defined by 
$$
F\left(x,y\right)=\frac{x-2y}{4}, \,(x,y)\in X^2.
$$ 
Then $F$ is mixed monotone and satisfies condition \eqref{Bhas1}  but  does not satisfy neither condition \eqref{eq-Luong} nor \eqref{Bhas}. 

Indeed, assume there exist $\varphi\in \Phi$ and $\psi\in \Psi$,  such that \eqref{eq-Luong} holds. This means that for all $x,y,u,v\in X$ with $\,x\geq u,\,y\leq v,$
$$
\left|\frac{x-2y}{4}-\frac{u-2v}{4}\right| \leq \frac{1}{2}\varphi\left(\left|x-u\right|+\left|y-v\right|\right)-\psi\left(\left|x-u\right|+\left|y-v\right|\right),
$$
which, in view of $(ii_{\varphi})$  yields, for $x=u$ and $y< v$,
$$
\frac{1}{2}\left|y-v\right| \leq \frac{1}{2}\varphi\left(\left|y-v\right|\right)-\psi\left(\left|y-v\right|\right)\leq \frac{1}{2}\varphi\left(\left|y-v\right|\right)<\frac{1}{2}\left|y-v\right|,
$$
a contradiction. Hence $F$ does not satisfy \eqref{eq-Luong}.

Now we prove that \eqref{Bhas1} holds. Indeed, since we have
$$
\left|\frac{x-2y}{4}-\frac{u-2v}{4}\right| \leq \frac{1}{4}\left|x-u\right|+\frac{1}{2}\left|y-v\right|,\,x\geq u,\,y\leq v,
$$
and
$$
\left|\frac{y-2x}{4}-\frac{v-2u}{4}\right| \leq \frac{1}{4}\left|y-v\right|+\frac{1}{2}\left|x-u\right|,\,x\geq u,\,y\leq v,
$$
by summing up the two inequalities above we get exactly  \eqref{Bhas1} with $\varphi(t)=t$ and $\psi(t)=\frac{1}{4} t$. Note also that $x_0=-2,\,y_0=3$ satisfy \eqref{mic}.

So  by our Theorem \ref{th3} we obtain that $F$ has a (unique) coupled fixed point $(0,0)$ but neither Theorem 2.1 in \cite{Luong} nor Theorem 2.1 in \cite{Bha} do not apply to $F$ in this example.
\end{example} 

\begin{remark} \em
Note also that Theorem 2.1 in \cite{Luong} has been proved under the additional very sharp condition on $\varphi$:
$$
(iii_{\varphi}) \,\,\varphi(s+t) \leq \varphi(s)+\varphi(t),\,\forall s,t \in [0,\infty),
$$
while our proof is independent of this assumption.
\end{remark}

\begin{corollary} \label{cor1}
	Let  $\left(X,\leq\right)$ be a partially ordered set and suppose there is a metric $d$  on $X$ such that $\left(X,d\right)$  is a complete metric space. Let $F : X \times X \rightarrow X $ be a  mixed monotone mapping for which there exists  $\psi \in \Psi$  such that for all $x,y,u,v\in X$ with $x \geq u, y \leq  v$,
$$
d\left(F\left(x,y\right),F\left(u,v\right)\right)+d\left(F\left(y,x\right),F\left(v,u\right)\right) \leq 
$$                                                                       
\begin{equation} \label{Bhas2}
 \leq  d\left(x,u\right) + d\left(y,v\right)-2 \psi\left(\frac{d\left(x,u\right) + d\left(y,v\right)}{2}\right).
\end{equation}  
Suppose either

(a) $F$ is continuous or

(b) $X$ satisfy Assumption \ref{1.1}.

If there exist $x_{0}, y_{0} \in X$ such that either \eqref{mic} or \eqref{mare} are satisfied, 
%\begin{equation} \label{mic1}
% x_{0} \leq F\left(x_{0},y_{0}\right)\textrm{ and }y_{0} \leq F\left(y_{0},x_{0}\right),
%\end{equation} or
%\begin{equation} \label{mare1} 
%x_{0} \geq F\left(x_{0},y_{0}\right)\textrm{ and }y_{0} \leq F\left(y_{0},x_{0}\right),
%\end{equation}
 then there exist $\overline{x}, \overline{y} \in X$ such that $$\overline{x} = F\left(\overline{x},\overline{y}\right)\textrm{and }\overline{y} = F\left(\overline{y},\overline{x}\right).$$ 

\end{corollary}

\begin{proof}
Taking $\varphi(t)=t$,  $t\in [0,\infty)$, condition \eqref{Bhas1} reduces to \eqref{Bhas2} and hence by Theorem \ref{th3} we get Corollary \ref{cor1}.
\end{proof}

\begin{remark} \em
If we take $\psi(t)=\left(1-\frac{k}{2}\right) t$, $t\in [0,\infty)$, with $0\leq k <1$, by Corollary \ref{cor1} we obtain a generalization of Theorem \ref{th2} ( Theorem 2.1 in \cite{Bha}). 
\end{remark}

\begin{remark} \em
Let us note that, as suggested by Example \ref{ex1},  since the contractivity condition \eqref{Bhas1} is valid  only for comparable elements in  $X^2$, Theorem \ref{th3} cannot guarantee in general the uniqueness of the coupled fixed point. 
\end{remark}

It is therefore our interest now to provide additional conditions to ensure that the coupled fixed point in Theorem \ref{th3} is in fact unique. Such a condition is the one used in Theorem 2.2 of Bhaskar and Lakshmikantham \cite{Bha} or in Theorem 2.4 of Luong and Thuan  \cite{Luong}:

%\begin{equation} \label{eq-6}
\textnormal{every pair of elements in} $X^2$ \textnormal{has either a lower bound or an upper bound},
%\end{equation}
which is known, see \cite{Bha}, to be equivalent to the following condition: for all $Y=(x,y),\,\overline{Y}=(\overline{x},\overline{y})\in X^2$,
\begin{equation} \label{eq-7}
\exists Z=(z_1,z_2)\in X^2 \textnormal{ that is comparable to }  Y \textnormal{ and } \overline{Y}.
\end{equation}

\begin{theorem} \label{th4}
In addition to  the  hypotheses  of Theorem \ref{th3}, suppose that condition  \eqref{eq-7} holds. Then $F$ has a
unique coupled fixed point.
\end{theorem}

\begin{proof} From Theorem \ref{th3}, the set of coupled fixed points of $F$ is nonempty. Assume that $Z^*=(x^*,y^*)\in X^2$ and $\overline{Z}=(\overline{x},\overline{y})$ are two coupled fixed point of $F$. We shall prove that $Z^*=\overline{Z}$.

By assumption \eqref{eq-7}, there exists $(u,v)\in X^2$ that is comparable to $(x^*,y^*)$ and $(\overline{x},\overline{y})$. We define the sequences $\{u_n\}$, $\{v_n\}$ as follows:
$$
u_0=u,\,v_0=v,\,u_{n+1}=F(u_n,v_n),\,v_{n+1}=F(v_n,u_n),\,n\geq 0.
$$
Since $(u,v)$ is comparable to $(\overline{x},\overline{y})$, we may assume $(\overline{x},\overline{y})\geq (u,v)=(u_0,v_0)$. By the proof of Theorem \ref{th3} we obtain inductively
\begin{equation} \label{eq-10}
(\overline{x},\overline{y})\geq (u_n,v_n), \, n\geq 0
\end{equation}
and therefore, by \eqref{Bhas1},
$$
\varphi\left(\frac{d(\overline{x},u_{n+1})+d(\overline{y},v_{n+1})}{2}\right)=
$$
$$
=\varphi\left(\frac{d(F(\overline{x},\overline{y}), F(u_n,v_n))+d(F(\overline{y}, \overline{x}),F(v_n,u_n))}{2}\right)
$$
\begin{equation} \label{eq-11}
\leq \varphi\left(\frac{d(\overline{x},u_{n})+d(\overline{y},v_{n})}{2}\right)-\psi\left(\frac{d(\overline{x},u_{n})+d(\overline{y},v_{n})}{2}\right),
\end{equation}
which, by the fact that $\psi\geq 0$, implies
$$
\varphi\left(\frac{d(\overline{x},u_{n+1})+d(\overline{y},v_{n+1})}{2}\right)\leq \varphi\left(\frac{d(\overline{x},u_{n})+d(\overline{y},v_{n})}{2}\right).
$$
Thus, by the monotonicity of $\varphi$, we obtain that the sequence $\{\Delta_n\}$ defined by
$$
\Delta_n=\frac{d(\overline{x},u_{n})+d(\overline{y},v_{n})}{2},\,n\geq 0,
$$
is non-increasing. Hence, there exists $\alpha\geq 0$ such that $\lim\limits_{n\rightarrow \infty} \Delta_n=\alpha$.

We shall prove that $\alpha=0$. Suppose, to the contrary, that $\alpha>0$. Letting $n\rightarrow \infty$ in \eqref{eq-11}, we get
$$
\varphi(\alpha) \leq \varphi(\alpha)-\lim_{n\rightarrow \infty} \psi(\Delta_n)=\varphi(\alpha)-\lim_{\Delta_n\rightarrow \alpha+} \psi(\Delta_n)<\varphi(\alpha).
$$
a contradiction. Thus $\alpha=0$, that is,
$$
\lim_{n\rightarrow \infty} \frac{d(\overline{x},u_{n})+d(\overline{y},v_{n})}{2}=0,
$$
which implies
$$
\lim_{n\rightarrow \infty} d(\overline{x},u_{n})=\lim_{n\rightarrow \infty}d(\overline{y},v_{n})=0.
$$
Similarly, we obtain that
$$
\lim_{n\rightarrow \infty} d(x^*,u_{n})=\lim_{n\rightarrow \infty}d(y^*,v_{n})=0,
$$
and hence $\overline{x}=x^*$ and $\overline{y}=y^*$.
\end{proof}

\begin{corollary} \label{cor2}
In addition to  the  hypotheses  of Corollary \ref{cor1}, suppose that condition  \eqref{eq-7} holds. Then $F$ has a
unique coupled fixed point.
\end{corollary}

An alternative uniqueness condition is given in the next theorem.

\begin{theorem} \label{th6}
In addition to the hypotheses of  Theorem \ref{th3}, suppose that $x_0,y_0 \in X$ are comparable. Then $F$ has a unique fixed point, that is, there exists $\overline{x}$ such that $F(\overline{x},\overline{x})=\overline{x}$.  
\end{theorem}

\begin{proof}
Assume we are in the case \eqref{mic}, that is 
$$
 x_{0} \leq F\left(x_{0},y_{0}\right)\textrm{ and }y_{0} \leq F\left(y_{0},x_{0}\right).
$$
Since $x_0,y_0$ are comparable, we have $x_0\leq y_0$ or $x_0\geq y_0$. Suppose we are in the second case. Then, by the mixed monotone property of $F$, we have
$$
 x_{1} =F\left(x_{0},y_{0}\right) \leq F\left(y_{0},x_{0}\right)=y_1,
$$
and, hence, by induction one obtains
\begin{equation} \label{eq-13}
x_n\geq y_n,\,n\geq 0.
\end{equation}
Now, since
$$
\overline{x}=\lim_{n\rightarrow\infty} F(x_n,y_n) \textnormal{ and } \overline{y}=\lim_{n\rightarrow\infty} F(y_n,x_n),
$$
by the continuity of the distance $d$, one has
$$
d(\overline{x},\overline{y})=d(\lim_{n\rightarrow\infty} F(x_n,y_n),\lim_{n\rightarrow\infty} F(y_n,x_n))=\lim_{n\rightarrow\infty} d(F(x_n,y_n), F(y_n,x_n))
$$
$$
=\lim_{n\rightarrow\infty} d(x_{n+1},y_{n+1}).
$$
On the other hand, by taking $Y=(x_n,y_n),\,V=(y_n,x_n)$ in \eqref{Bhas1} we have
$$
\varphi(d(F(x_n,y_n),F(y_n,x_n)))\leq \varphi( d(x_n,y_n))-\psi( d(x_n,y_n)),\,n\geq 0,
$$
which actually means
$$
\varphi(d(x_{n+1},y_{n+1}))\leq  \varphi( d(x_n,y_n))-\psi( d(x_n,y_n)),\,n\geq 0.
$$
Suppose $\overline{x}\neq \overline{y}$, that is $d(\overline{x},\overline{y})>0$. Taking the limit as $n\rightarrow\infty$ in the previous inequality, we get
$$
\varphi(d(\overline{x},\overline{y}))=\lim_{n\rightarrow\infty} \varphi (d(x_{n+1},y_{n+1}))\leq \varphi(d(\overline{x},\overline{y}))-\lim_{n\rightarrow\infty} \psi( d(x_n,y_n)),
$$
or 
$$
\lim_{d(x_n,y_n)\rightarrow d(\overline{x},\overline{y})} \psi( d(x_n,y_n))\leq 0,
$$
which contradicts $(i_{\psi})$. Thus $\overline{x}= \overline{y}$.
\end{proof}

\begin{remark} \em
Note that in \cite{Bha} and \cite{Luong} the authors use only condition \eqref{mic}, although the alternative assumption \eqref{mare} is also acceptable.
\end{remark}

\section{Application to integral equations}
As an application of the (coupled) fixed point theorems established in Section 2 of our paper, we study the existence and uniqueness of the solution to a Fredholm nonlinear integral equation.

In order to compare our  results  to the ones in \cite{Luong}, we shall consider the same integral equation, that is,
\begin{equation}\label{3.1}
x(t)=\int_{a}^{b} \left(K_1(t,s)+K_2(t,s)\right)\left(f(s,x(s))+g(s,x(s))\right)ds+h(t),
\end{equation}
$t\in I=[a,b].$

Let $\Theta$ denote the set of all functions $\theta:[0,\infty)\rightarrow [0,\infty)$ satisfying 

$(i_\theta)$ $\theta$ is non-decreasing;

$(ii_\theta)$ There exists $\psi\in \Psi$ such that $\theta(r)=\frac{r}{2}-\psi\left(\frac{r}{2}\right)$, for all $r\in [0,\infty)$. 
\smallskip

As shown in \cite{Luong}, $\Theta$ is nonempty, as $\theta_1(r)=k r$ with $0\leq 2 k<1$; $\theta_2(r)=\frac{r^2}{2(r+1)}$; and $\theta_3(r)=\frac{r}{2}-\frac{\ln(r+1)}{2}$, are all elements of $\Theta$.

Like in \cite{Luong}, we assume that the functions $K_1,K_2,f, g$  fulfill the following conditions:
\begin{ass}

$(i)$ $K_1(t,s)\geq 0$ and $K_2(t,s)\leq 0$, for all $t,s\in I$;

$(ii)$ There exist the positive numbers $\lambda, \mu$, such that for all $x,y\in\R$, with $x\geq y$, the following Lipschitzian type conditions hold:
\begin{equation}\label{3.3}
   0\leq f(t,x)-f(t,y)\leq \lambda \theta(x-y) 
\end{equation}
and
\begin{equation}\label{3.4}
    -\mu \theta(x-y)\leq g(t,x)-g(t,y)\leq 0;
\end{equation}
$(iii)$ 
\begin{equation}\label{3.4-1}
(\lambda+\mu)\cdot \sup_{t\in I}\int_{a}^{b}\left[K_1(t,s)-K_2(t,s)\right]ds \leq 1.%\eqno(*)
\end{equation}
\end{ass}

\begin{definition}
\em (\cite{Luong}) A pair  $(\alpha, \beta)\in X^2$ with $X=C(I,\mathbb{R})$ is called a \emph{coupled lower-upper solution} of equation (\ref{3.1})  if, for all $t\in I$,
$$
  \alpha(t)\leq  \int_{a}^{b}K_1(t,s)\left[f(s,\alpha(s))+g(s,\beta(s))\right]ds+
$$
$$
+\int_{a}^{b}K_2(t,s)\left[f(s,\beta(s))+g(s,\alpha(s))\right]ds+h(t)
$$
and
$$
 \beta(t)\geq  \int_{a}^{b}K_1(t,s)\left[f(s,\beta(s))+g(s,\alpha(s))\right]ds+
$$
$$
+\int_{a}^{b}K_2(t,s)\left[f(s,\alpha(s))+g(s,\beta(s))\right]ds+h(t),
$$

\end{definition}

\begin{theorem} \label{th3.1}
Consider the integral equation \eqref{3.1} with 
$$
K_1,K_2 \in C(I\times I,\mathbb{R}) \textnormal{ and } h\in C(I,\mathbb{R}).
$$
Suppose that there exists a coupled lower-upper solution of \eqref{3.1} and that Assumption 3.1 is satisfied. Then the integral equation \eqref{3.1} has a unique solution in $C(I,\mathbb{R})$.
\end{theorem}
\begin{proof}
Consider on $X=C(I,\mathbb{R})$ the natural partial order relation, that is, for $x,y\in C(I,\mathbb{R})$
$$
x\leq y\Leftrightarrow x(t)\leq y(t),\,\forall t\in I.
$$
It is well known that $X$ is a complete metric space with respect to the sup metric
$$
d(x,y)=\sup_{t\in I}\left|x(t)-y(t)\right|,\,x,y\in C(I,\mathbb{R}).
$$
Now define on $X^2$ the following partial order: for $(x,y),(u,v)\in X^2$,
$$
(x,y)\leq (u,v)\Leftrightarrow x(t)\leq u(t), \textnormal{ and } y(t)\geq v(t)\,\forall t\in I.
$$
Obviously, for any $(x,y)\in X^2$, the functions $\max\{x ,y\}$, $\min\{x ,y\}$ are the upper and lower bounds of $x,y$, respectively.

Therefore, for every $(x,y),(u,v)\in X^2$, there exists the element $(\max\{x ,y\}, \min\{x ,y\})$ which is comparable to $(x,y)$ and $(u,v)$.

Define now the mapping $F:X\times X\rightarrow X$ by
$$
F(x,y)(t)=\int_{a}^{b}K_1(t,s)\left[f(s,x(s))+g(s,y(s))\right]ds+
$$
$$
+\int_{a}^{b}K_2(t,s)\left[f(s,y(s))+g(s,x(s))\right]ds+h(t), \,\textnormal{ for all }t\in I.
$$
It is not difficult to prove, like in \cite{Luong}, that $F$ has the mixed monotone property. Now for $x,y,u,v \in X$ with $x\geq u$ and $y\leq v$, we have
$$
d(F(x,y),F(u,v))=\sup_{t\in I}\left|F(x,y)(t)-F(u,v)(t)\right|=
$$
$$
=\sup_{t\in I}\left|\int_{a}^{b}K_1(t,s)\left[f(s,x(s))+g(s,y(s))\right]ds+\right.
$$
$$
+\int_{a}^{b}K_2(t,s)\left[f(s,y(s))+g(s,x(s))\right]ds-
$$
$$
-\int_{a}^{b}K_1(t,s)\left[f(s,u(s))+g(s,v(s))\right]ds-
$$
$$
\left.-\int_{a}^{b}K_2(t,s)\left[f(s,v(s))+g(s,u(s))\right]ds\right|=
$$
$$
=\sup_{t\in I}\left|\int_{a}^{b}K_1(t,s)\left[f(s,x(s))-f(s,u(s))+g(s,y(s))-g(s,v(s))\right]ds+\right.
$$
$$
\left.+\int_{a}^{b}K_2(t,s)\left[f(s,y(s))-f(s,v(s))+g(s,x(s))-g(s,u(s))\right]ds\right|=
$$
$$
=\sup_{t\in I}\left|\int_{a}^{b}K_1(t,s)\left[(f(s,x(s))-f(s,u(s)))-(g(s,v(s))-g(s,y(s)))\right]ds\right.
$$
$$
\left.-\int_{a}^{b}K_2(t,s)\left[(f(s,v(s))-f(s,y(s)))-(g(s,x(s))-g(s,u(s)))\right]ds\right|\leq
$$
$$
\leq \sup_{t\in I}\left|\int_{a}^{b}K_1(t,s)\left[\lambda \theta(x(s)-u(s))+\mu \theta(v(s)-y(s))\right]ds-\right.
$$
\begin{equation} \label{3.2}
\left.-\int_{a}^{b}K_2(t,s)\left[\lambda \theta(v(s)-y(s))+\mu \theta(x(s)-u(s))\right]ds\right|.
\end{equation}
Since the function $\theta$ is non-decreasing and $x\geq u$ and $y\leq v$, we have
$$
\theta(x(s)-u(s))\leq \theta(\sup_{t\in I}\left|x(t)-u(t)\right|=\theta(d(x,u))
$$
and
$$
\theta(v(s)-y(s))\leq \theta(\sup_{t\in I}\left|v(t)-y(t)\right|=\theta(d(v,y)),
$$
hence by \eqref{3.2}, in view of the fact that $K_2(t,s)\leq 0$, we obtain
$$
d(F(x,y),F(u,v))\leq \sup_{t\in I}\left|\int_{a}^{b}K_1(t,s)\left[\lambda \theta(d(x,u))+\mu \theta(d(v,y))\right]ds-\right.
$$
$$
\left.-\int_{a}^{b}K_2(t,s)\left[\lambda \theta(d(v,y)+\mu \theta(d(x,u))\right]ds\right|=
$$
$$
=\left[\lambda \theta(d(x,u))+\mu \theta(d(v,y))\right]\cdot \sup_{t\in I}\left|\int_{a}^{b}\left[K_1(t,s)-K_2(t,s)\right]ds\right|=
$$
\begin{equation} \label{3.5}
=\left[\lambda \theta(d(x,u))+\mu \theta(d(v,y))\right]\cdot \sup_{t\in I}\int_{a}^{b}\left[K_1(t,s)-K_2(t,s)\right]ds,
\end{equation}
since $K_2(t,s)\leq 0$. 
Similarly, we obtain
$$
d(F(y,x),F(v,u))\leq
$$
\begin{equation} \label{3.6}
=\left[\lambda \theta(d(v,y))+\mu \theta(d(x,u))\right]\cdot \sup_{t\in I}\int_{a}^{b}\left[K_1(t,s)-K_2(t,s)\right]ds.
\end{equation}
By summing \eqref{3.5} and \eqref{3.6} we get, by using \eqref{3.4-1},
$$
\frac{d(F(x,y),F(u,v))+d(F(y,x),F(v,u))}{2}\leq (\lambda+\mu)\cdot 
$$
$$
\cdot \sup_{t\in I}\int_{a}^{b}\left[K_1(t,s)-K_2(t,s)\right]ds\cdot \frac{\theta(d(v,y))+ \theta(d(x,u))}{2}\leq 
$$
$$
\leq \frac{\theta(d(v,y))+ \theta(d(x,u))}{2}.
$$
 Now, since $\theta$ is non-increasing, we have
$$
\theta(d(x,u))\leq \theta(d(x,u)+d(v,y)),\,\theta(d(v,y))\leq \theta(d(x,u)+d(v,y))
$$
and so
$$
\frac{\theta(d(v,y))+ \theta(d(x,u))}{2}\leq \theta(d(x,u)+d(v,y))=
$$
$$
=\frac{d(v,y)+ d(x,u)}{2}-\psi\left(\frac{d(v,y)+ d(x,u)}{2}\right),
$$
by the definition of $\theta$. Thus we finally get
$$
\frac{d(F(x,y),F(u,v))+d(F(y,x),F(v,u))}{2}\leq 
$$
$$
=\frac{d(v,y)+ d(x,u)}{2}-\psi\left(\frac{d(v,y)+ d(x,u)}{2}\right).
$$
which is just the contractive condition \eqref{Bhas2} in Corollary \ref{cor1}.

Now, let $(\alpha, \beta)\in X^2$ be a coupled upper-lower solution of (\ref{3.1}). Then we have
$$
  \alpha(t)\leq F(\alpha(t), \beta(t)) \textnormal{ and } \beta(t)\geq F[\beta(t), \alpha(t)),  
$$
for all $t\in I$, which show that all hypotheses of Corollary \ref{cor1} are satisfied.\\
This proves that $F$ has a unique coupled fixed point $(\overline{x},\overline{y})$ in $X^2$.

Since $\alpha\leq \beta$, by Corollary \ref{cor2}  it follows that $\overline{x}=\overline{y}$, that is
$$
\overline{x}=F(\overline{x},\overline{x}),
$$
and therefore $\overline{x}\in C(I,\mathbb{R})$ is the unique solution of the integral equation \eqref{3.1}.
\end{proof}

\begin{remark}\em
Note that our Theorem \ref{th3.1} is more general than Theorem 3.3 in \cite{Luong} since, if $\lambda\neq \mu$, then
$$  
\lambda+\mu < 2 \max \{\lambda, \mu\}.
$$
For example, if in Assumption 3.1 we have $\lambda=\frac{1}{6}$, $\mu=\frac{1}{12}$ and\\ $\sup_{t\in I}\int_{a}^{b}\left[K_1(t,s)-K_2(t,s)\right]ds=4$, then our condition \eqref{3.4-1} holds:
$$
(\lambda+\mu)\cdot \sup_{t\in I}\int_{a}^{b}\left[K_1(t,s)-K_2(t,s)\right]ds =\frac{1}{4}\cdot 4\leq 1,  
$$
so Theorem \ref{th3.1} can be applied but, since
$$  
2\max\{\lambda,\mu\}\cdot \sup_{t\in I}\int_{a}^{b}\left[K_1(t,s)-K_2(t,s)\right]ds =2\cdot\frac{1}{6}\cdot 4=\frac{4}{3}> 1,  
$$
the corresponding condition $(iii)$ in \cite{Luong} does not hold and hence Theorem 3.3 in \cite{Luong} cannot be applied to obtain the existence and uniqueness of the solution of the integral equation \eqref{3.1}.
\end{remark}

\begin{remark}\em
As a final conclusion, we note that our results in this paper improve all coupled fixed point theorems in \cite{Bha}-\cite{Luong}, as well as the fixed point theorems in \cite{Aga}, \cite{Nie06}-\cite{Ran}, by considering a more general (symmetric) contractive condition. Note also that our technique of proof reveals that one can use the dual assumption \eqref{mare}  for the initial values $x_0,y_0$ in Theorem \ref{th3}. %Moreover, while the results in \cite{Luong} are obtained under the sharp assumption "$\varphi$ is subadditive", in our proofs we do not make use of it. 
\end{remark}

\vskip 0.5 cm {\it 

Department of Mathematics and Computer Science

North University of Baia Mare

Victoriei 76, 430122 Baia Mare ROMANIA

E-mail: vberinde@ubm.ro}

\begin{thebibliography}{00}

\bibitem{Aga}  Agarwal, R.P., El-Gebeily, M.A. and O'Regan, D., \textit{Generalized contractions in partially ordered metric spaces}, Appl. Anal. \textbf{87} (2008) 1–-8

\bibitem{Ber07}  Berinde, V., \textit{Iterative approximation of fixed points}. Second edition, Lecture Notes in Mathematics, 1912, Springer, Berlin, 2007

\bibitem{Bha}  Bhaskar, T. G., Lakshmikantham, V., \textit{Fixed point theorems in partially ordered metric spaces and applications}, Nonlinear Anal. \textbf{65} (2006), no. 7, 1379--1393 
 
\bibitem{LakC} Lakshmikantham, V., \' Ciri\' c, L.,  \textit{Coupled fixed point theorems for nonlinear contractions in partially ordered metric spaces}, Nonlinear Anal. \textbf{70} (2009), 4341–-4349

\bibitem{Luong} Luong N. V., Thuan, N. X., \textit{Coupled fixed points in partially ordered metric spaces and application}, Nonlinear Anal.,  \textbf{74} (2011), 983--992

\bibitem{Nie06} Nieto, J. J., Rodriguez-Lopez, R.,\ \textit{Contractive mapping theorems in partially ordered sets and applications to ordinary differential  equations}, Order \textbf{22} (2005), no. 3, 223-239 (2006) 
 
\bibitem{Nie07} Nieto, J. J., Rodriguez-Lopez, R.,\ \textit{Existence and uniqueness of fixed point in partially ordered sets  and applications to ordinary differential  equations}, Acta. Math. Sin., (Engl. Ser.) \textbf{23}(2007), no. 12, 2205--2212
 
\bibitem{Ran} Ran, A. C. M., Reurings, M. C. B., \textit{A fixed point theorem in partially ordered sets and some applications to matrix equations}, Proc. Amer. Math. Soc. \textbf{132} (2004), no. 5, 1435--1443
  
%\bibitem {Rus1} Rus, I. A., \textit{Generalized Contractions and Applications}, Cluj University Press, Cluj-Napoca, 2001              
\bibitem {Rus2}  Rus, I. A., Petru\c sel, A., Petru\c sel, G., \textit{Fixed Point Theory}, Cluj University Press, Cluj-Napoca, 2008

\end{thebibliography}
\end{document}